\documentclass{amsart}
\usepackage{amsfonts,amssymb,amscd,amsmath,enumerate,verbatim,calc}
\usepackage[all]{xy}

\newcommand{\CM}{Cohen-Macaulay}

\newcommand{\wrt}{with respect to}

\newcommand{\n}{\mathfrak{n} }
\newcommand{\m}{\mathfrak{m} }
\newcommand{\M}{\mathfrak{M} }

\newcommand{\Bf}{\mathbf{f} }
\newcommand{\Bg}{\mathbf{g} }

\newcommand{\Z}{\mathbb{Z} }

\newcommand{\rt}{\rightarrow}

\newcommand{\ov}{\overline}

\newcommand{\wt}{\widetilde }

\newcommand{\ord}{\operatorname{ord}}
\newcommand{\grade}{\operatorname{grade}}
\newcommand{\depth}{\operatorname{depth}}
\newcommand{\projdim}{\operatorname{projdim}}
\newcommand{\ann}{\operatorname{ann}}

\newcommand{\Tor}{\operatorname{Tor}}
\newcommand{\vpd}{\operatorname{vpd}}
\newcommand{\embdim}{\operatorname{embdim}}
\newcommand{\codim}{\operatorname{codim}}

\newcommand{\cx}{\operatorname{cx}}

\newcommand{\Ext}{\operatorname{Ext}}
\newcommand{\Syz}{\operatorname{Syz}}

\theoremstyle{plain}

\newtheorem{thm}{Theorem}

\newtheorem{theorem}{Theorem}[section]

\newtheorem{lemma}[theorem]{Lemma}
\newtheorem{proposition}[theorem]{Proposition}

\theoremstyle{definition}
\newtheorem{definition}[theorem]{Definition}
\newtheorem{s}[theorem]{}
\newtheorem{remark}[theorem]{Remark}
\newtheorem{example}[theorem]{Example}

\theoremstyle{remark}

\begin{document}

\title[Hilbert function of MCM modules]{ The Hilbert function of a maximal Cohen-Macaulay module\\ Part II}
\author{Tony~J.~Puthenpurakal}
\date{\today}
\address{Department of Mathematics, IIT Bombay, Powai, Mumbai 400 076}

\email{tputhen@math.iitb.ac.in}

\subjclass{Primary 13A30; Secondary 13D40, 13D07}

 \begin{abstract}Let $(A,\m)$ be a strict complete intersection of positive dimension and let $M$ be a maximal \CM \ $A$-module with  bounded betti-numbers.  We prove that the Hilbert function of $M$ is non-decreasing. We also prove an analogous statement for complete intersections of codimension two.
\end{abstract}
 \maketitle
\section*{introduction}
Let $(A,\m)$ be a $d$-dimensional Noetherian ring  with residue field $k$ and let $M$ be a finitely generated $A$-module. Let $\mu(M)$ denote minimal number of generators of $M$ and let $\ell(M)$ denote its length. Let $\codim(A) = \mu(\m) - d$ denote the codimension of $A$.

Let $G(A) = \bigoplus_{n\geq 0}\m^n/\m^{n+1}$ be the associated graded ring of $A$ (\wrt \ $\m$) and let $G(M) = \bigoplus_{\n\geq 0}\m^nM/\m^{n+1}M$ be the associated graded module of $M$ considered as a $G(A)$-module. The ring $G(A)$ has a  unique graded maximal ideal
$\M_G =  \bigoplus_{n\geq 1}\m^n/\m^{n+1} $.
 Set $ \depth G(M) =  \grade(\M_G,G(M))$.

The Hilbert function of $M$ (\wrt \ $\m$) is the function
\[
H(M,n) = \ell \left( \frac{\m^nM}{\m^{n+1}M} \right) \quad \text{for all} \ n\geq 0.
\]
It is clear that if $\depth G(M) > 0$ then the Hilbert function of $M$ is \emph{non} decreasing, see Proposition 3.2 of \cite{Pu2}.
If $A$ is regular local then all maximal \CM \ (= MCM ) modules are free. Thus every MCM module of positive dimension over a regular local ring has a non-decreasing Hilbert function. The next case is that
of a hypersurface ring i.e., the completion $\widehat{A} = Q/(f)$ where $(Q,\n)$ is regular local and $f \in \n^2$.
In part 1 of this paper we proved that if $A$ is a hypersurface ring of positive dimension and if $M$ is a MCM $A$-module then the Hilbert function of $M$ is non-decreasing (see Theorem 1,\cite{Pu2}). See example 3.3 of part 1 of the paper for an example of a MCM module $M$ over the hypersurface ring
$k[[x,y]]/(y^3)$ with $\depth G(M) = 0$.

In the ring case Elias \cite[2.3]{Elias},  proved that the Hilbert function of a  one dimensional Cohen-Macaulay ring is non-decreasing if embedding dimension is three. The first example of a  one dimensional \CM \ ring $A$ with not monotone increasing Hilbert function was given by Herzog and Waldi; \cite[3d]{HW}. Later Orecchia, \cite[3.10]{Ore}, proved that for all $b \geq 5$ there exists a reduced one-dimensional \CM \ local ring of embedding dimension $b$  whose Hilbert function is not monotone increasing.
Finally in \cite{GR}  we can find similar example with embedding dimension four. A long standing conjecture in theory of Hilbert functions is that
the Hilbert function of a one dimensional  complete intersection (of positive dimension) is non-decreasing. Rossi conjectures that a similar result holds for
Gorenstein rings.

In this paper we investigate Hilbert function of MCM modules over complete intersection rings; ie., rings $A$ with completion $\widehat{A} = Q/ (f_1,\ldots,f_c)$ where $(Q,\n)$ is regular local and $f_1,\ldots, f_c$ is a $Q$-regular sequence and $f_i \in  \n^2$ for each $i$.   A direct generalization of the above result is false.
See \ref{ex1}  for an example of strict complete intersection $A$ with a MCM  module $M$ such that the
Hilbert function of $M$ is not monotone. Recall a local ring $(A, \m) $ is said to be a strict complete intersection if $A$ is a complete intersection and
$G(A)$ is also a complete intersection.

To describe our result we need to recall the notion of complexity of a module.
 This notion was introduced by Avramov in \cite{LLAV}.
 Let $\beta_i^A(M) = \ell( \Tor^A_i(M,k) )$ be the $i^{th}$ Betti number of $M$ over $A$. The complexity of $M$ over $A$ is defined by
\[
\cx_A M = \inf\left \lbrace b \in \mathbb{N}  \left\vert \right.   \varlimsup_{n \to \infty} \frac{\beta^A_n(M)}{n^{b-1}}  < \infty \right \rbrace.
\]
If $A$ is a local complete intersection of $\codim c$ then $\cx_A M \leq c$. Furthermore all values between $0$ and $c$ occur.
Note that $\cx_A M \leq 1$ if and only if $M$ has bounded Betti numbers.

 Let $A$ be a hypersurface ring and $M$  a non-free MCM $A$-module. Then
 $$\Syz^{A}_{n+2}(M) \cong \Syz^A_{n}(M)  \quad \text{for all} n \geq 1. $$
  This result has been generalized to complete intersection of arbitrary codimension by Eisenbud, i.e., if $A$ is a local complete intersection and if $M$ is a MCM
$A$-module with bounded Betti-numbers then $\Syz^{A}_{n+2}(M) \cong \Syz^A_{n}(M)$ for all $n \geq 1$; see \cite{Eis}. A hypersurface ring is also a strict complete intersection. Our generalization of Theorem 1 in \cite{Pu2} is the following:
\begin{thm}\label{strict}
Let $(A,\m)$ be a  strict complete intersection  ring of positive dimension.
Let  $M$ be a maximal \CM \ $A$-module with bounded Betti numbers. Then the Hilbert function of $M$ is non-decreasing.
\end{thm}

There does exist complete intersection local ring $A$ of positive dimension and codimension two
with $G(A) $ having depth zero; for instance see \ref{ex2}.
By one of  our earlier result the Hilbert function of $A$ is non-decreasing. More generally we show
\begin{thm}\label{cd2}
Let $(A,\m)$ be a complete intersection local ring of positive dimension and codimension two.
Let  $M$ be a maximal \CM \ $A$-module with bounded Betti numbers. Then the Hilbert function of $M$ is non-decreasing.
\end{thm}

Here is an overview of the contents of the paper. In section one we discuss a few preliminary facts that we need. In this section we discuss Eisenbud operator's. The proof of Theorem 1 and 2 involves the notion of virtual projective dimension; see \cite{LLAV}. In section two we prove Theorem 2. In section three we give a proof of Theorem 1.

\section{Preliminaries}
In this paper all rings are Noetherian and all modules considered are assumed to be finitely generated
\begin{remark}
Let $x_1,... ,x_s$ be a sequence in  $\m$ and
set $J= (x_1,... ,x_s)$.
Set $B = A/J$,  $\n = \m/J$ and  $N = M/JM$. Notice
\[
G_{\m}(N) = G_{\n}(N) \quad \text{and} \quad \depth_{ G(A)} G(N) = \depth_{G(B)} G(N).
\]
\end{remark}

\s \textbf{Base change:}
\label{AtoA'}
 Let $\phi \colon (A,\m) \rt (A',\m')$ be a local ring homomorphism. Assume
   $A'$ is a faithfully flat $A$
algebra with $\m A' = \m'$. Set $\m' = \m A'$ and if
 $N$ is an $A$-module set $N' = N\otimes_A A'$.
 In these cases it can be seen that

\begin{enumerate}[\rm (1)]
\item
$\ell_A(N) = \ell_{A'}(N')$.
\item
 $H(M,n) = H(M',n)$ for all $n \geq 0$.
\item
$\dim M = \dim M'$ and  $\depth_A M = \depth_{A'} M'$.
\item
$\depth G(M) = \depth G(M')$.
\item
$A'$ is a (strict) local complete intersection if and only if $A$ is a (strict) local complete intersection.
\end{enumerate}

 \noindent The specific base changes we do are the following:

(i) $A' = A[X]_S$ where $S =  A[X]\setminus \m A[X]$.
The maximal ideal of $A'$ is $\n = \m A'$.
The residue
field of $A'$ is $K = k(X)$. We do this base change when the residue field $k$ is finite.

(ii) $A' = \widehat{A}$ the completion of $A$ with respect to the maximal ideal.

Thus we can assume that our ring $A$ is complete with infinite residue field.

\s To prove Theorem's \ref{strict} and \ref{cd2} we need the notion of cohomological operators over a complete intersection ring; see \cite{Gull} and
\cite{Eis}.
 Let $\mathbf{f} = f_1,\ldots,f_c$ be a regular sequence in a  local Noetherian ring $Q$. Set $I = (\mathbf{f})$ and
 $ A = Q/I$,
\s

The \emph{Eisenbud operators}, \cite{Eis}  are constructed as follows: \\
Let $\mathbb{F} \colon \cdots \rightarrow F_{i+2} \xrightarrow{\partial} F_{i+1} \xrightarrow{\partial} F_i \rightarrow \cdots$ be a complex of free
$A$-modules.

\emph{Step 1:} Choose a sequence of free $Q$-modules $\wt{F}_i$ and maps $\wt{\partial}$ between them:
\[
\wt{\mathbb{F}} \colon \cdots \rightarrow \wt{F}_{i+2} \xrightarrow{\wt{\partial}} \wt{F}_{i+1} \xrightarrow{\wt{\partial}} \wt{F}_i \rightarrow \cdots
\]
so that $\mathbb{F} = A\otimes\wt{\mathbb{F}}$

\emph{Step 2:} Since $\wt{\partial}^2 \equiv 0 \ \text{modulo} \ (\mathbf{f})$, we may write  $\wt{\partial}^2  = \sum_{j= 1}^{c} f_j\wt{t}_j$ where
$\wt{t_j} \colon \wt{F}_i \rightarrow \wt{F}_{i-2}$ are linear maps for every $i$.

 \emph{Step 3:}
Define, for $j = 1,\ldots,c$ the map $t_j = t_j(Q, \mathbf{f},\mathbb{F}) \colon \mathbb{F} \rightarrow \mathbb{F}(-2)$ by $t_j = A\otimes\wt{t}_j$.

\s
The operators $t_1,\ldots,t_c$ are called Eisenbud's operator's (associated to $\mathbf{f}$) .  It can be shown that
\begin{enumerate}
\item
$t_i$ are uniquely determined up to homotopy.
\item
$t_i, t_j$ commute up to homotopy.
\end{enumerate}

\s \label{Basis-change} Let $\mathbf{g} = g_1,\ldots, g_c$ be another regular sequence in $Q$ and assume $I = (\Bf) = (\Bg)$. The Eisenbud operator's associated to
$\Bg$ can be constructed by using the corresponding operators associated to $\Bf$. An explicit construction is as follows:
Let
$$f_i = \alpha_{i1}g_1 + \cdots + \alpha_{ic} g_c \quad \text{for} \ i = 1, \ldots, c. $$
Then we can choose
\begin{equation}\label{bchange-eqn}
t_i^\prime = \alpha_{1i}t_1 + \cdots + \alpha_{ci} t_c \quad \text{for} \ i = 1, \ldots, c.
\end{equation}
as Eisenbud operators for $g_1,\ldots,g_c$. In \cite{Eis} there is a mistake in indexing. So we give a full proof here.

For the proof
it is convenient to use matrices.
Here $\alpha = (\alpha_{ij})$ is a $c\times c$ invertible matrix with coefficients in $Q$. If $ \phi = (\phi_{ij})$ be a $m \times n$ matrix with coefficients in $Q$ then we set $\phi^{tr} = (\phi_{ji})$ to be the transpose of $\phi$. Set $[\Bf]$ and $[\Bg]$ to be the column vectors
$(f_1,\ldots,f_c)^{tr}$ and $(g_1,\ldots,g_c)^{tr}$ respectively. Thus in matrix terms we have
\[
[\Bf ] = \alpha\cdot [\Bg].
\]
Let $t_1, \ldots, t_c$ be the operators associated to $\Bf$ and let $\wt{t_1}, \ldots, \wt{t_c}$ be as above. By hypothesis we have
$\sum_{j=1}^{c} f_j \wt{t_j} = \wt{\partial^2}$. Set $[\mathbf{\wt{t}}] = ( \wt{t_1}, \ldots, \wt{t_c})^{tr}$. Define $[\mathbf{\wt{t}}^\prime] = ( \wt{t_1}^\prime, \ldots, \wt{t_c}^\prime)^{tr}$ by
\[
[\mathbf{\wt{t}^\prime}] = \alpha^{tr}\cdot [\mathbf{\wt{t}}].
\]
We prove
\begin{proposition}\label{bcL}
(with hypothesis as in \ref{Basis-change})
\[
\sum_{j=1}^{c} g_j\wt{t_j}^\prime = \wt{\partial^2}.
\]
\end{proposition}
\begin{proof}
\begin{align*}
\sum_{j=1}^{c} g_j\wt{t_j}^\prime &= [\Bg]^{tr}\cdot [\mathbf{\wt{t}^\prime}]  \\
&= [\Bg]^{tr}\cdot \left(  \alpha^{tr}\cdot [\mathbf{\wt{t}}] \right) \\
                                &= \left( [\Bg]^{tr} \cdot \alpha^{tr} \right) \cdot [\mathbf{\wt{t}}] \\
                                 &= \left( \alpha \cdot [\Bg] \right)^{tr} \cdot [\mathbf{\wt{t}}] \\
                                &= [\Bf]^{tr}\cdot  [\mathbf{\wt{t}}]   \\
                                       &= \wt{\partial^2}.
\end{align*}
\end{proof}
Define, for $j = 1,\ldots,c$ the map $t_j^\prime = \wt{t_j}^\prime \otimes A$. Then $t_1^\prime,\cdots,t_c^\prime$ are the Eisenbud operators associated to $\Bg$.
Furthermore we have
\[
[\mathbf{t^\prime}] = \alpha^{tr}\cdot [\mathbf{t}].
\]

\s Let $R = A[t_1,\ldots,t_c]$ be a polynomial ring over $A$ with variables $t_1,\ldots,t_c$ of degree $2$. Let $M, N$ be  finitely generated $A$-modules. By considering a free resolution $\mathbb{F}$ of $M$ we get well defined maps
\[
t_j \colon \Ext^{n}_{A}(M,N) \rightarrow \Ext^{n+2}_{R}(M,N) \quad \ \text{for} \ 1 \leq j \leq c  \ \text{and all} \  n,
\]
which turn $\Ext_A^*(M,N) = \bigoplus_{i \geq 0} \Ext^i_A(M,N)$ into a module over $R$. Furthermore these structure depend only on $\Bf$, are natural in both module arguments and commute with the connecting maps induced by short exact sequences.

\s  Gulliksen, \cite[3.1]{Gull},  proved that if $\projdim_Q M$ is finite then
$\Ext_A^*(M,N) $ is a finitely generated $R$-module. For $N = k$, the residue field of $A$, Avramov in \cite[3.10]{LLAV} proved a converse; i.e., if
$\Ext_A^*(M,k)$ is a finitely generated $R$-module then $\projdim_Q M$ is finite. For a more general result, see \cite[4.2]{AGP}.

\s Since $\m \subseteq \ann \Ext^{i}_A(M,k)$ for all $i \geq 0$ we get that $\Ext^*_A(M,k)$ is a module over $S = R/\m R = k[t_1,\ldots,t_c]$.
If $\projdim_Q M$ is finite then $\Ext^*_A(M,k)$ is a finitely generated $S$-module of Krull dimension $\cx M$.

\s \label{inverse}
Going mod $\m$ we get that the ring $S$ is invariant of a minimal generating set of $I = (\Bf)$.
Conversely if $\xi_1,\ldots,\xi_c \in S_2$ be such that $S = k[\xi_1,\ldots,\xi_c]$ then there exist's  a regular sequence
$\Bg = g_1,\ldots,g_c$ such that
\begin{enumerate}
\item
$(\Bg) = (\Bf)$
\item
if $t_j^\prime$ are the Eisenbud operators associated to $g_j$ for $j = 1,\ldots,c$ then the action of $t_j^\prime$ on $\Ext^*_A(M,k)$ is same as that
of $\xi_j$ for $j = 1,\ldots,c$.
\end{enumerate}
This can be seen as follows.
Let
\[
\xi_i = \ov{\beta_{i1}}t_1 + \cdots + \ov{\beta_{ic}}t_c \quad \text{for} \ i = 1,\ldots,c.
\]
Let $\beta = (\beta_{ij}) \in M_c(R)$. Note that $\beta$ is an invertible matrix since $\ov{\beta} = (\ov{\beta_{ij}})$ is an invertible matrix in $M_n(k)$. Set $\alpha = \beta^{tr}$. Define $\Bg = g_1,\ldots,g_c$ by
\[
[\Bg] = \alpha^{-1}\cdot[\Bf]
\]
Clearly $\Bg = g_1,\ldots,g_c$ is a regular sequence and $(\Bf) = (\Bg)$.
Notice $[\Bf] = \alpha \cdot [\Bg]$. So by \ref{Basis-change} the cohomological operators $t_1^{\prime},\ldots, t_c^{\prime}$ associated to $\Bg$
is given by the formula
\[
[\mathbf{t^\prime}] = \alpha^{tr}\cdot [\mathbf{t}] = \beta \cdot [\mathbf{t}].
\]
It follows that the action of $t_j^\prime$ on $\Ext^*_A(M,k)$ is same as that
of $\xi_j$ for $j = 1,\ldots,c$.

The following result is easy to prove.

\begin{lemma}\label{choice}
Let $K$ be an infinite field and let $S = K[X_1,\ldots,X_n]$ be a polynomial ring with $\deg X_i = 2$ for each $i = 1,\ldots,n$. Let
$E = \bigoplus_{i \in \Z}E_i$ be a finitely generated graded $S$-module of Krull dimension one.
Then there exists $\xi = \beta_1 X_1 + \cdots + \beta_n X_n$ where $\alpha_i \in K$ such that
\begin{enumerate}
\item
$\xi$ is a parameter for $E$.
\item
$\beta_1, \beta_2, \ldots, \beta_n$ are all non-zero
\end{enumerate}
\qed
\end{lemma}

\section{Proof of Theorem \ref{cd2}}
In this section we assume that $(A,\m)$ is a complete local ring with infinite residue field;  see \ref{AtoA'}. Recall that a local ring
$(Q,\n)$ is a (codimension c) \emph{deformation} of $A$, if there exists a surjective map $\rho \colon Q \rightarrow A$ with $\ker \rho$ generated by a $Q$-regular sequence of length $c$. The deformation is called \emph{embedded} if $\ker \rho \subseteq \n^2$. Given a deformation of $A$, we view every $A$-module as a $Q$-module, via $\rho$.

\begin{definition}
For an $A$-module $M$, the \emph{virtual projective dimension}, $\vpd_A M$,  is (the non-negative integer or $\infty$)
$$ \vpd_A M = \min \{ \projdim_Q M \mid Q \ \text{is a deformation of $A$} \}.                $$
\end{definition}
In Lemma 3.4(3) of \cite{LLAV} it is proved that
$$ \vpd_A M = \min \{ \projdim_Q M \mid Q \ \text{is an embedded  deformation of $A$} \}.                $$
Furthermore in Theorem 3.5 of \cite{LLAV} Avramov proves that
\[
\vpd_A M = \depth A - \depth M + \cx_A M.
\]
\begin{proof}[Proof of Theorem \ref{cd2}]
We may without  loss of generality assume that $A$ is complete and has an infinite residue field.
Note that
\[
\vpd_A M = \depth A - \depth M + \cx_A M  = \cx_A M \leq 1
\]
The first case we consider is when $\vpd_A M = 0$. In  this case $M$ has finite projective dimension over $A$. As $M$ is also maximal \CM, we get that $M$ is free and so by our earlier result the Hilbert function of $M$ is non-decreasing.

The next case is when  $\vpd_A M = 1$. Note that
\[
\vpd_A M =  \min \big\{ \projdim_Q M \mid Q  \  \text{  is an embedded deformation of } \  A     \big \}
\]
Thus there exists an embedded deformation $(Q_0, \n_0)$ of $A$ such that
\[
\projdim_{Q_0} M = 1.
\]
By the Auslander-Buchsbaum formula we get
\[
\depth_{Q_0} M + \projdim_{Q_0} M = \depth Q_0.
\]
Notice that $Q_0$ is \CM. Furthermore
\[
\depth_{Q_0} M  = \depth_A M = \dim M = \dim A = d.
\]
Thus we have $\dim Q_0 = d+1$. Furthermore as $Q_0$ is an embedded deformation of $A$ we get that
\[
\embdim Q_0 = \embdim A = d + 2
\]
It follows that $Q_0$ is a local hypersurface ring of dimension $d+1 \geq 2$.
Notice that $G(Q_0)$ is \CM \ of dimension  $\geq 2$. Furthermore $\projdim_{Q_0} M = 1$. We now use a result from part 1 of the paper, see
 Theorem \ref{chief}, to get that the Hilbert function of $M$ is non-decreasing.
\end{proof}
For the reader's convenient we state Theorem 5 of part 1 of this paper.
\begin{theorem}\label{chief}
Let $(Q,\n)$ be a Noetherian local ring and let $M$ be a finitely generated $Q$-module with $\projdim_Q M \leq 1$. If $\depth G(Q) \geq 2$ then
the Hilbert function of $M$ is non-decreasing.
\end{theorem}
In \cite{VallaB} there is an example of a complete intersection $A$  of codimension 2 with $\depth G(A) = 0$. For the convenience of the
reader we give it here.
\begin{example} \label{ex2}
Let $K$ be a field and let $A = K[[t^6, t^7, t^{15}]]$. It can be verified that
\[
A \cong \frac{K[[X,Y,Z]]}{(Y^3-XZ, X^5 - Z^2)}
\]
and that
\[
G(A) \cong \frac{K[X,Y,Z]}{(XZ, Y^6, Y^3Z,Z^2)}
\]
Note that $ZY^2$ annihilates $(X,Y,Z)$. So $\depth G(A) = 0$.
\end{example}

\section{Proof of Theorem \ref{strict}}
In this section we give a proof of Theorem 1. We also give an example of a maximal \CM \ module over a strict complete intersection such that the Hilbert function of $M$ is not monotone increasing.
\begin{theorem}
Let $(A,\m)$  be a strict complete intersection ring of dimension $d \geq 1$. Let $M$ be a maximal \CM \ $A$-module with $\cx_A M \leq 1$. Then the
Hilbert function of $M$ is non-decreasing.
\end{theorem}
\begin{proof}
By 1.2 we may assume that $A$ is complete and has an infinite residue field. So
 $A = Q/(\Bf)$ where $(Q,\n)$ is a regular local ring with infinite residue field, $\Bf = f_1,\ldots,f_c \in \n^2$
is a $Q$-regular sequence. Furthermore we may assume $\Bf^* = f_1^*,\ldots,f_c^*$ is a $G(Q)$-regular sequence and $G(A) = G(Q)/(\Bf^*)$.

For $x \in \n^i \setminus \n^{i+1} $ we set $\ord(x) = i$.
Without loss of any generality we may further assume that
\[
\ord(f_1) \geq \ord(f_2) \geq \cdots \geq \ord(f_c).
\]

If $\cx_A M = 0$ then $M$ is free. It follows that the Hilbert function of $M$ is non-decreasing. So we now assume that $\cx_A M = 1$.
As discussed in 1.10,  $\Ext^*_A(M,k)$ is a finitely generated graded $S = k[t_1,\ldots,t_c]$ module of Krull dimension one. By Lemma \ref{choice}
there exists $\xi_1 = \ov{\beta_1}t_1 + \cdots + \ov{\beta_c}t_c \in S_2$ such that
\begin{enumerate}
\item
$\xi_1$ is a parameter for $E$.
\item
$\ov{\beta_1}, \ov{\beta_2}, \ldots, \ov{\beta_c}$ are all non-zero.
\end{enumerate}
Set $\xi_j = t_j$ for $j = 2,\ldots,c$. Then $S = k[\xi_1, \ldots, \xi_c]$.

Set $\beta = (\beta_{ij})$ where
\[
\beta_{ij}  =
\begin{cases}
\beta_j, &\text{if $i = 1$;} \\
1, &\text{if $i = j$ and $i \geq 2$;} \\
0, &\text{otherwise.}
\end{cases}
\]
Clearly $\beta$ is an invertible matrix in $M_n(A)$. By \ref{inverse} there exists a regular sequence $\Bg = g_1, \ldots,g_c$ such that
\begin{enumerate}
\item
$(\Bg) = (\Bf)$
\item
if $t_j^\prime$ are the Eisenbud operators associated to $g_j$ for $j = 1,\ldots,c$ then the action of $t_j^\prime$ on $\Ext^*_A(M,k)$ is same as that
of $\xi_j$ for $j = 1,\ldots,c$.
\end{enumerate}
By \ref{inverse}
\[
[\Bg] = (\beta^{tr})^{-1}\cdot [\Bf] = (\beta^{-1})^{tr} \cdot [\Bf].
\]
It is easy to compute the inverse of $\beta$. So we obtain
\begin{align*}
g_1 &= \frac{1}{\beta_1}f_1 \\
g_j &= \frac{-\beta_j}{\beta_1} f_1 + f_j \ \text{for} \ j = 2,\ldots,c.
\end{align*}
Recall that $\ord(f_1) \geq \ord(f_j)$ for $j = 2, \ldots,c$. Notice that if $\ord(f_1) = \ord(f_j)$ then
\[
\frac{-\beta_j}{\beta_1} f_1^* + f_j^* \neq 0,
\]
since $f_1^*, \ldots, f_c^*$ is a $G(Q)$-regular sequence. Thus we have
\[
g_1^* = \frac{1}{\beta_1}f_1^*  \quad \text{and}
\]
for $2 \leq j \leq c$ we have
\[
g_j^* =
\begin{cases}
f_j^*, &\text{if $\ord(f_1) > \ord(f_j)$;} \\
\frac{-\beta_j}{\beta_1} f_1^* + f_j^*, &\text{if $\ord(f_1) = \ord (f_j)$.}
\end{cases}
\]
Notice that $\Bg^* = g_1^*,\ldots,g_c^*$ is a regular sequence in $G(Q)$ and $(\Bg^*) = (\Bf^*)$. The subring
$k[\xi_1]$ of $S$ can be identified with the ring $S^\prime$ of cohomology operators of a presentation $A = P/(g_1)$ where
$P = Q/(g_2,\ldots,g_c)$. Since $\Ext_A^*(M,k)$ is a finite module over $S^\prime$ we get that $\projdim_P M$ is finite.

Notice that
\begin{enumerate}
\item
$\dim P = \dim A + 1 \geq 2$.
\item
$\projdim_P M = 1$.
\item
$P$ is a strict complete intersection, since $G(P) = G(Q)/(g_2^*,\ldots,g_c^*)$.
\item
$G(P)$ is \CM \ of dimension $\geq 2$.
\end{enumerate}
It follows from Theorem \ref{chief} that the Hilbert function of $M$ is non-decreasing.
\end{proof}
We give an example of an MCM module over a strict Gorenstein ring having non-monotone Hilbert function.
\begin{example}\label{ex1}
Let $(B,\n)$ be a \CM \ local ring of positive dimension having non-monotone Hilbert function. By \cite[2.5]{PuGor} there exists a strict complete intersection $A$ with $\dim A = \dim B$ such that $B$ is a quotient of $A$. Clearly $B$ is a MCM $A$-module.
\end{example}



\providecommand{\bysame}{\leavevmode\hbox to3em{\hrulefill}\thinspace}
\providecommand{\MR}{\relax\ifhmode\unskip\space\fi MR }
\providecommand{\MRhref}[2]{%
  \href{http://www.ams.org/mathscinet-getitem?mr=#1}{#2}
}
\providecommand{\href}[2]{#2}

\end{document}